\documentclass[10pt]{amsart}

\usepackage[colorlinks]{hyperref}
\usepackage{color,graphicx,shortvrb}
\usepackage[latin 1]{inputenc}
\usepackage{multicol}

\usepackage[active]{srcltx} %SRC Specials for DVI search

\usepackage{enumerate}

\newtheorem{theorem}{Theorem}[section]

\newtheorem{corollary}[theorem]{Corollary}

\newtheorem{problem}[theorem]{Problem}
\newtheorem{proposition}[theorem]{Proposition}

\def\J#1#2#3{ \left\{ #1,#2,#3 \right\} }

\def\11{\textbf{$1$}}
\def\CC{{\mathbb{C}}}

\def\TT{{\mathbb{T}}}

\DeclareGraphicsExtensions{.jpg,.pdf,.png,.eps}

\usepackage{enumerate}

\begin{document}

\title[The linear biholomorphic property for von Neumann algebra preduals]{von Neumann algebra preduals satisfy the linear biholomorphic property}
\date{}

\author[A.M. Peralta]{Antonio M. Peralta}
\email{aperalta@ugr.es}
\address{Departamento de An{\'a}lisis Matem{\'a}tico, Facultad de
Ciencias, Universidad de Granada, 18071 Granada, Spain.}

\author[L.L. Stachó]{Laszlo L. Stachó}
\email{stacho@math.u-szeged.hu}
\address{Bolyai Institute, University of Szeged, 6720 Szeged, HUNGARY}

\thanks{First author partially supported by the Spanish Ministry of Economy and Competitiveness,
D.G.I. project no. MTM2011-23843, and Junta de Andaluc\'{\i}a grant FQM3737.}



\date{August, 2013}

\maketitle
 \thispagestyle{empty}

\begin{abstract} We prove that for every JBW$^*$-triple $E$ of rank $>1$, the symmetric part of its predual reduces to zero. Consequently, the predual of every infinite dimensional von Neumann algebra $A$ satisfies the linear biholomorphic property, that is, the symmetric part of $A_*$ is zero. This solves a problem posed by M. Neal and B. Russo \cite[to appear in \emph{Mathematica Scandinavica}]{NealRu12}.
\end{abstract}

\section{Introduction}

The open unit ball of every complex Banach space satisfies certain holomorphic properties which determine the global isometric structure of the whole space. An illustrative example is the following result of W. Kaup and H. Upmeier \cite{KaUp77}.

\begin{theorem}\cite{KaUp77}\label{t Kaup Upmeier 77}
Two complex Banach spaces whose open unit balls are biholomorphically equivalent are linearly isometric.$\hfill\Box$
\end{theorem}

We recall that, given a domain $U$ in a complex Banach space $X$ (i.e. an open, connected subset), a function $f$ from $U$ to another complex Banach space $F$ is said to be \emph{holomorphic} if the Fréchet derivative of $f$ exists at every point in $U$. When $f : U \to f(U)$ is holomorphic and bijective, $f (U)$ is open in $F$ and $f^{-l} : f (U) \to U$ is holomorphic, the mapping $f$ is said to be \emph{biholomorphic}, and the sets $U$ and $f (U)$ are \emph{biholomorphically
equivalent}. Theorem \ref{t Kaup Upmeier 77} gives an idea of the power of infinite-dimensional Holomorphy in Functional Analysis. A reviewed proof of Theorem \ref{t Kaup Upmeier 77} was published by J. Arazy in \cite{Arazy87}.\smallskip

A consequence of the results established by Kaup and Upmeier in \cite{KaUp77} gave raise to the study of the symmetric part of an arbitrary complex Banach space in the following sense: Let $X$ be a complex Banach space with open unit ball denoted by $D$. Let $G = {\rm Aut} (D)$ denote the group of all biholomorphic automorphisms of $D$ and let $G^O$ stand for the connected component of the identity in $G$. Given a holomorphic function $h : D \to X$, we can define a holomorphic vector field $Z= h (z) \frac{\partial}{\partial z}$, which is a composition differential operator on the space $H (D ,X)$ of all holomorphic functions from $D$ to $X,$ given by $X(f)(z) = (h(z) \frac{\partial}{\partial z}) f (z) =f'(z) (h(z))$, ($z\in D$). It is known that, for each $z_0$ the initial value problem $\frac{\partial}{\partial t} \varphi (t, z_0) = h(\varphi (t, z_0))$, $\varphi(0; z_0) = z_0$ has a unique solution $\varphi (t, z_0) : J_{z_0} \to D$
defined on a maximal open interval $J_{z_0}\subseteq \mathbb{R}$ containing $0$. The holomorphic mapping $h$ is called \emph{complete} when $J_{z_0}= \mathbb{R}$, for every $z_0\in D$. Denoting by ${\rm aut} (D)$ the Lie algebra of all complete, holomorphic vector fields on $D$, the \emph{symmetric part of $D$} is
$D_S = G(O) = G^O(O)$. The \emph{symmetric part of $X$}, denoted by $X_S$ or by $S(X)$, is the orbit of $0$ under the set ${\rm aut} (D)$ of all complete holomorphic vector fields on $D$. Furthermore, $X_S$ is a closed, complex subspace of $X$, $D_S = X_S \cap D$, and hence, $D_S$ is the open unit ball of $X_S$, $D_S$ is \emph{symmetric} in the sense that for each $z\in D_S$ there exists a symmetry of $D$ at $z$, i.e., a mapping $s_z \in {\rm Aut} (D_S)$ such that
$s_z (z)=z$, $s_z^2=identity$, and $s_z' (z)=-Id_E$; thus $D_S =E_S \cap D$ is a \emph{bounded symmetric domain} (cf. \cite{KaUp77}, \cite{BraKaUp78}, and \cite{Arazy87}).\smallskip

A Jordan structure associated with the symmetric part of every complex Banach space $X$ was also determined by W. Kaup and H. Upmeier in \cite{KaUp77}. Namely, for every $a \in X_S$ there is a unique symmetric continuous bilinear mapping $Q_a : X \times X \to X$ such that $(a - Q_a (z,z)) \frac{\partial}{\partial z}$ is a complete holomorphic vector field on $D$. A partial triple product is defined on $X\times X_S \times X$ by the assignment $$\{.,.,.\} : X\times X_S \times X \to X,$$ $$ \J xay := Q_a (x,y).$$ It is known (cf. \cite{KaUp77} and \cite{BraKaUp78}) that the partial triple product satisfies the following properties:
\begin{enumerate}[$(i)$]\label{axioms partial triple product} \item $\J ...$ is bilinear and symmetric in the outer variables and conjugate linear in the middle one;
\item $\J {X_S}{X_S}{X_S} \subseteq {X_S}$;
\item The Jordan identity $$\J ab{\J xyz} = \J {\J abx}yz - \J x{\J bay}z + \J xy{\J abz},$$ holds for every $a,b,y\in X_S$ and $x,z\in X;$
\item For each $a\in X_S$, the mapping $L(a,a): X\to X$, $z\mapsto \J aaz$ is a hermitian operator;
\item The identity $\J {\J xax}bx =  \J xa{\J xbx}$ holds for every $a,b\in X_S$ and $x\in X$.
\end{enumerate} It should be remarked here that property $(v)$ appears only implicitly in \cite{BraKaUp78}. A complete substantiation is included in \cite{Panou} (compare also \cite{Sta91}).\smallskip

The extreme possibilities for the symmetric part $X_S$ (i.e. $X_S =X$ or $X_S =\{0\}$) define particular and significant classes of complex Banach spaces. The deeply studied class of \emph{JB$^*$-triples}, introduced by W. Kaup in \cite{Ka83}, is exactly the class of those complex Banach spaces $X$ for which $X_S =X$. In the opposite side, we find the complex Banach spaces satisfying the \emph{linear biholomorphic property} (LBP, for short). A complex Banach space $X$ with open unit ball $D$ satisfies the LBP when its symmetric part is trivial (cf. \cite[page 145]{Arazy87}).\smallskip

The symmetric part of some classical Banach spaces was studied and determined by R. Braun, W. Kaup and H. Upmeier \cite{BraKaUp78}, L.L. Stachó \cite{Sta79}, J. Arazy \cite{Arazy87}, and J. Arazy and B. Solel \cite{ArazySolel90}. The following list covers the known cases:\begin{enumerate}[$(i)$]\item For $X = L_p(\Omega, \mu)$, $1 \leq p < 1$, $p \neq 2$, and dim$(X)\geq 2$, we have $X_S = 0$;
\item For $X =  H_p$ the classical Hardy spaces with $1 \leq p < 1$, $p \neq 2$, we have $X_S = 0$;
\item For $X = H_{\infty}$ or the disk algebra, $X_S = \mathbb{C}$;
\item When $X$ is a uniform algebra $A \subseteq C(K)$, $A_S = A \cap \overline{A}$.
\item When $A$ is a subalgebra of $B(H)$ containing the identity operator $I$, then $A_S$ is the maximal C$^*$-subalgebra $A \cap A^*$ of $A$;
\item Let $X$ be a complex Banach space with a 1-unconditional basis. Then $X= X_S$ if and only if $X$ is the $c_0$-sum of a sequence of Hilbert spaces. Moreover, if $X$ is a symmetric sequence space (i.e. the unit vector basis form a 1-symmetric basis of E) then either $X_S=\{0\}$ or $X_S=X$. In the last case, either $X=\ell_1$ or $X = c_0$.
\end{enumerate}

In a very recent contribution, M. Neal and B. Russo stated the following problem:

\begin{problem}\cite[Problem 2]{NealRu12}\label{problem NealRusso} Is the symmetric part of the predual of a von Neumann algebra equal
to 0? What about the predual of a JBW$^*$-triple which does not contain a Hilbert space as a direct summand?
\end{problem}

In this note we give a complete answer to the questions posed by Neal and Russo \cite{NealRu12} in the above problem. Our main result proves that for every JBW$^*$-triple $W$ which is not a Hilbert space, the symmetric part of its predual reduces to zero. In particular the symmetric part of the predual of an infinite-dimensional von Neumann algebra is equal to $\{0\}$. Unfortunately, there exist examples of JBW$^*$-triples $W$ containing a Hilbert space as a direct summand for which $S(W_*)= (W_*)_S = \{0\}$.

\section{Computing the symmetric part of a JBW$^*$-triple predual}

We recall that a JB$^*$-triple is a complex Banach space $E$ satisfying that $E_S =E$. JB$^*$-triples were introduced by W. Kaup in \cite{Ka83}, where he also gave the following axiomatic definition of these spaces: A \emph{JB$^*$-triple} is a complex Banach space $E$ equipped with a
triple product $\{\cdot,\cdot,\cdot\}:E\times E\times E\rightarrow
E$ which is linear and symmetric in the outer variables, conjugate
linear in the middle one, satisfies the axioms $(iii)$ and $(iv)$ in \eqref{axioms partial triple product} and
the following condition:
\begin{enumerate}[$(vi)$]
\item $\|\{x,x,x\}\|=\|x\|^3$ for all $x\in E$.
\end{enumerate}

Every C$^*$-algebra is a complex JB$^*$-triple with respect to the
triple product \linebreak $\J xyz = \frac{1}{2} ( x y^* z + z y^*
x)$, and in the same way every JB$^*$-algebra with respect to $\J
abc = \left( a\circ b^{*}\right) \circ c+\left( c\circ
b^{*}\right) \circ a-\left( a\circ c\right) \circ b^{*}$.\smallskip

Elements $a,b$ in a JB$^*$-triple $E$ are said to be \emph{orthogonal} (denoted by $a\perp b$) whenever $L(a,b) =0$. It is known that $a\perp b$ $\Leftrightarrow \J aab=0$ $\Leftrightarrow \J bba=0$ (cf. \cite[Lemma 1]{BurFerGarMarPe}). The rank, $r(E)$, of a real or complex JB$^*$-triple $E$, is the minimal cardinal number $r$ satisfying $card(S)\leq r$ whenever $S$ is an orthogonal subset of $E$, i.e. $0\notin S$ and $x\perp y$ for every $x\neq y$ in $S$.\smallskip

We briefly recall that an element $e$ in a JB$^*$-triple $E$ is said to be a \emph{tripotent} whenever $\J eee =e$. A tripotent $e\in E$ is said to be complete whenever $a\perp e$ implies $a=0$. When the condition $\{e,e,a\} = a$ implies that $a\in \mathbb{C} e$, we shall say that $e$ is a minimal tripotent. The symbol Tri$(E)$ will stand for the set of all tripotents in $E$.\smallskip

The following characterization of complete holomorphic vector fields, which is originally due to L.L. Stachó (see \cite{Sta79}, \cite{Sta82} and \cite{Up87}), has been borrowed from \cite[Proposition 2.5]{ArazySolel90}.

\begin{proposition}\label{p Stacho} Let $X$ be a complex Banach space whose open unit ball is denoted by $D$ and let $h: D \to  X$ be a holomorphic mapping. Then $h \in {\rm aut}(D)$ if and only if $h$ extends holomorphically to a neighborhood of $\overline{D}$, and, for every $z\in X$, $\varphi\in X^*$ satisfying $\|z\| =\| \varphi\| =1 = \varphi (z)$, we have $\Re{\rm e} \varphi(h(z)) = 0$.$\hfill\Box$
\end{proposition}

In order to simplify the arguments, we recall some geometric notions. Elements $s,y$ in a complex Banach space $X$ are said to be \emph{$L$-orthogonal}, denoted by $x\perp_{L} y$, (respectively, \emph{$M$-orthogonal}, denoted by $x\perp_{M} y$) if $\| x\pm y\| = \|x\|+\|y\|$ (respectively, $\| x\pm y\| = \max \{\|x\|,\|y\|\}$). It is known that $x\perp_L y $ if, and only if, for all real numbers $s,t$, $s x\perp_L ty$ if, and only if, there exist elements $a,b\in X^*$ satisfying $a\perp_M b$, $\| x\| \ \| a\| = \|x\| = a(x)$, and $\| y\| \ \| b\| = \|y\| = b(y)$ (see, for example, \cite[Lemma 3.1 and Corollary 4.3]{EdRu01}). It is also known that for each pair of elements $(a,b)$ in a JB$^*$-triple $E$, the condition $a\perp b$ implies $a\perp_M b$ (cf. \cite[Lemma 1]{BurFerGarMarPe} and \cite[Lemma 1.3(a)]{FriRu85}).\smallskip

We also recall that a JBW$^*$-triple is a JB$^*$-triple which is also a dual Banach space. In this sense, JBW$^*$-triples play an analogue role to that given to von Neumann algebras in the setting of C$^*$-algebras. Every JBW$^*$-triple admits a unique (isometric) predual and its product is separately weak$^*$-continuous (see \cite{BarTi}).\smallskip

We can proceed with a first technical result on the structure of the symmetric part of a JBW$^*$-triple predual.

\begin{proposition}\label{p orthoognal tripotents annihilate elements in the symmetric part} Let $W$ be a JBW$^*$-triple with predual $W_*=F$.
Suppose, $e_1,e_2$ are two tripotent elements in $W$, $\varphi_1,\varphi_2\in F$ with $\Vert \varphi_k\Vert =1$, $e_1\perp e_2$, and $e_j(\varphi_k)=\delta_{jk}$ $(j,k=1,2)$. Then $e_1 (\phi)= e_2 (\phi)=0$, for every $\phi$ in $F_S$.
\end{proposition}

\begin{proof} Let $\phi$ be an element in $F_S$. %The symbol $\{.,.,.\}$ will indistinctly stand for the partial triple product of $F$ and for the triple product of $W$. 
Since $\phi\in F_S$, the holomorphic vector field $\big[ \phi - Q_\phi (z,z) \big] \frac{\partial}{\partial z}$ is tangent to the unit sphere of $F$. Thus, by Proposition \ref{p Stacho},
$$\Re{\rm e} \big\langle e,  \phi - Q_\phi (\varphi,\varphi) \big\rangle =0,$$ for every $\varphi\in F$, $e\in W$ with $\Vert \varphi \Vert = \Vert e\Vert =1 =
\big\langle e,\varphi \big\rangle \big(=e(\varphi)\big).$\smallskip

Since $e_1\perp e_2$ implies $e_1\perp_M e_2$, it follows from the hypothesis that $\varphi_1 \perp_L \varphi_2$. In particular, for any weight $0\leq \lambda\leq 1$ and $\kappa_1,\kappa_2\in\mathbb{T} :=\{ \kappa\in\CC:\ \vert \kappa\vert=1\}$, $\kappa_1 (1-\lambda) \varphi_1 + \kappa_2\lambda \varphi_2$ belongs to the unit sphere of $F$ and $\overline{\kappa_1} e_1 + \overline{\kappa_2} e_2$ is a supporting functional for it. Therefore,
$$0 = \Re{\rm e} \Big\langle \overline{\kappa_1} e_1 + \overline{\kappa_2} e_2 , \phi - Q_\phi\big(
\kappa_1 (1-\lambda) \varphi_1 + \kappa_2\lambda \varphi_2,\kappa_1 (1-\lambda) \varphi_1 + \kappa_2\lambda \varphi_2\big)
\Big\rangle $$
$$= \Re{\rm e} \Big( \overline{\kappa_1} e_1(\phi) + \overline{\kappa_2} e_2(\phi) +
\kappa_1(1-\lambda)^2 \alpha_1 + \kappa_2\lambda^2 \alpha_2 +
\kappa_1\lambda(1-\lambda)\beta_1 + \kappa_2\lambda(1-\lambda)\beta_2\Big) $$
with the constants $\alpha_k:= \big\langle e_k, Q_\phi(\varphi_k,\varphi_k) \big\rangle$,
$\beta_k:= 2 \big\langle e_{3-k} , Q_\phi(\varphi_k,\varphi_{3-k})\big\rangle $. % $= 2 \big\langle e_{3-k} , Q_a(\varphi_{3-k},\varphi_k)\big\rangle$.
In particular, with the choice $\lambda=1$ we get
$$\Re{\rm e}\big( \overline{\kappa_1} e_1(\phi) + \overline{\kappa_2} e_2(\phi) + \kappa_2 \alpha_2 \big) =0$$ for every $\kappa_1,\kappa_2\in\TT$. Replacing $\kappa_2$ with $-\kappa_2$ we have $\Re {\rm e} \big( \kappa_1 e_1 (\phi) \big)=0 \quad (\kappa_1\in\TT),$ and hence $e_1(\phi)=0$. %With a the index change $1\!\leftrightarrow\! 2$ for $k$ we get $e_2(a)=0$.
\end{proof}

Before dealing with our main result we shall review some results on JB$^*$-triples of rank one. For a JB$^*$-triple $E$,
the following are equivalent:
\begin{enumerate}[$(a)$] \item $E$ has rank one;
\item $E$ is a complex Hilbert space equipped with the triple product given by $2 \{ a,b,c\} :=  (a|b) c +(c|b) a$, where $(.|.)$ denotes the inner product of $E$;
\item The set of complete tripotents in $E$ is non-zero and every complete tripotent in $E$ is minimal;
\item $E$ contains a complete tripotent which is minimal.
\end{enumerate}

The equivalence $(a)\Leftrightarrow (b)$ follows, for example, from \cite[Proposition 4.5]{BuChu}. The implications $(b)\Rightarrow (c)$ and $(c)\Rightarrow (d)$ are clear. It should be commented here that a general JB$^*$-triple might not contain any tripotent. However, since the complete tripotents of a JB$^*$-triple $E$ coincide with the real and complex extreme points of its closed unit ball (cf. \cite[Proposition 3.5]{KaUp77b} and \cite[Lemma 4.1]{BraKaUp78b}), by the Krein-Milman theorem, every JBW$^*$-triple contains an abundant set of (complete) tripotents. In the setting of JBW$^*$-triples, a tripotent $e$ is minimal if and only if it cannot be written as an orthogonal sum of two (non-zero) tripotents (compare the arguments in \cite[Proposition 2.2]{PeSta}). Back to the equivalences, the implication $(d)\Rightarrow (a)$ is established in \cite[Proposition 3.7 and its proof]{BurGarPe11}.

\begin{theorem}\label{t symmetric predual} Let $W$ be a JBW$^*$-triple of rank $>1$ and let $F$ denote its predual. Then $F_S = \{0\}$, that is, $F$  satisfies the linear biholomorphic property.
\end{theorem}

\begin{proof} Let $\phi$ be an element in $F_S$. According to the Krein-Milman Theorem, the finite linear combinations of the extreme points of
the closed unit ball, $D(W)$, of $W$ form a weak$^*$-dense subset in $D(W)$. Therefore, it suffices to prove that \begin{equation}\label{eq 1 thm}
e(\phi)=0 \quad \hbox{for all} \ \ e\in{\rm Ext}\big( \overline{D}(W)\big),
\end{equation} or equivalently, $e(\phi)=0$ for every complete tripotent $e\in W$.

Let $e$ be a complete tripotent in $W$. Since $W$ has rank $>1$, the comments preceding this theorem guarantee the existence of two non-zero tripotents $e_1, e_2$ in $W$ such that $e_1\perp e_2$ and $e =  e_1+ e_2$. Let us notice that the JBW$^*$-subtriple $U$ of $W$ generated by $e_1 $ and $e_2$ coincides with $\mathbb{C} e_1 \bigoplus^{\infty} \mathbb{C} e_2.$ We can easily define two norm-one functionals $\psi_1,\psi_2$ in $U_*$ satisfying $\psi_j (e_k)=\delta_{jk}$. By \cite[Theorem]{Bun01}, there exists norm-one weak$^*$-continuous functionals $\varphi_1, \varphi_2$ in $W_*$ which are norm-preserving extensions of $\psi_1$ and $\psi_2$, respectively. Applying Proposition \ref{p orthoognal tripotents annihilate elements in the symmetric part} we have $e_j (\phi)=0$, for every $j=1,2$, and finally $e(\phi) = e_1 (\phi) + e_2 (\phi)=0$ as we desired.\end{proof}

It is known that a von Neumann algebra, regarded as a JBW$^*$-triple, has rank one if and only if it coincides with $\mathbb{C}$. We therefore have:

\begin{corollary}\label{c von Neumann preduals} Let $W$ be a von Neumann algebra of dimension $<1$ and let $F=W_*$. Then $F_S = \{0\}$, that is, $F$ satisfies the linear biholomorphic property.$\hfill\Box$
\end{corollary}

There is an additional aspect of Problem \ref{problem NealRusso} that should be commented. Suppose $H$ is a complex Hilbert space, $W$ is a non-zero JBW$^*$-triple, and consider the JBW$^*$-triple $U= H \bigoplus^{\infty} W$ (the orthogonal sum of $H$ and $W$). %Every tripotent $e$ in $V$ is of the form $e=(e_1,e_2)$ where $e_1$ is a norm-one element in $H$ or $e_1=0$ and $e_2$ is a tripotent in $W$. It can be easily checked that $e$ is complete if and only if $\|e_1\|=1$ and $e_2$ is complete.
It is clear that $U$ has rank $>1$. Thus, Theorem \ref{t symmetric predual} implies that $S(U_*)=\{0\}.$ In other words, the predual of a JBW$^*$-triple which does not contain a Hilbert space as a direct summand satisfies the linear biholomorphic property but the class of all JBW$^*$-triples whose preduals satisfy the linear biholomorphic property is strictly bigger.

\bigskip


\begin{thebibliography}{10}
\bibitem{Arazy87} J. Arazy, An application of infinite-dimensional holomorphy to the geometry of Banach spaces, in \emph{Geometrical aspects of Functional Analysis (1985/86)},  122-150, Lecture Notes in Math., 1267, Springer, Berlin, 1987.

\bibitem{ArazySolel90} J. Arazy, B. Solel, Isometries of nonselfadjoint operator algebras, \emph{J. Funct. Anal.} \textbf{90}, no. 2, 284-305 (1990).

\bibitem{BarTi} T.J. Barton, R.M. Timoney, Weak$^*$-continuity of Jordan triple products and its applications. Math. Scand. {\bf 59}, 177-191 (1986).

\bibitem{BraKaUp78} R. Braun, W. Kaup, H. Upmeier, On the automorphisms of circular and Reinhardt domains in complex Banach spaces, \emph{Manuscripta Math.} \textbf{25}, no. 2, 97-133 (1978).

\bibitem{BraKaUp78b} R. Braun, W. Kaup, H. Upmeier, A holomorphic characterization of Jordan C$^*$-algebras, \emph{Math. Z.} \textbf{161}, 277-290 (1978).

\bibitem{Bun01} L.J. Bunce, Norm preserving extensions in JBW$^*$-triple, \emph{Quart. J. Math. Oxford} \textbf{52}, No.2, 133-136 (2001).

\bibitem{BuChu} L.J. Bunce and C.-H. Chu,  Compact  operations, multipliers and Radon-Nikodym property in $JB^*$-triples, \emph{Pacific J. Math.} \textbf{153}, 249-265 (1992).

\bibitem{BurFerGarMarPe} M. Burgos, F.J. Fern{\'a}ndez-Polo, J. Garc{\'e}s, J.
Mart{\'\i}nez, A.M. Peralta, Orthogonality preservers in
C$^*$-algebras, JB$^*$-algebras and JB$^*$-triples, \emph{J. Math. Anal.
Appl.} \textbf{348}, 220-233 (2008).

\bibitem{BurGarPe11} M. Burgos, J.J. Garc{\'e}s, A.M. Peralta, Automatic continuity of biorthogonality preservers between weakly compact JB$^*$-triples and atomic JBW$^*$-triples, \emph{Studia Math.} \textbf{204} (2) 97-121 (2011).

\bibitem{EdRu01} C.M. Edwards, G.T. R{\"u}ttimann, Orthogonal faces of the unit ball in a Banach space, \emph{Atti Sem. Mat. Fis. Univ. Modena} \textbf{49}, 473-493 (2001).

\bibitem{FriRu85} Y. Friedman, B. Russo, Structure of the predual of a JBW$^*$-triple, \emph{J. Reine Angew. Math.} {\bf 356}, 67-89 (1985).

%\bibitem{Ho87} G. Horn, Characterization of the predual and ideal structure of a ${\rm JBW}^*$-triple, \emph{Math. Scand.} \textbf{61}, no. 1, 117-133 (1987).

%\bibitem{IsKaRo95} J.M. Isidro, W. Kaup, A. Rodr{\'\i}guez, On real forms of JB$^*$-triples, \emph{Manuscripta Math.} \textbf{86}, 311-335 (1995).

\bibitem{Ka83} W. Kaup, A Riemann Mapping Theorem for bounded symmentric domains in complex Banach spaces, \emph{Math. Z.} \textbf{183}, 503-529 (1983).

%\bibitem{Ka2} W. Kaup, Contractive projections on Jordan C$^*$-algebras and generalizations, \textit{Math. Scand.} \textbf{54}, 95-100 (1984).

%\bibitem{Ka96} W. Kaup, On spectral and
%singular values in JB$^*$-triples, \emph{Proc. Roy. Irish Acad. Sect.
%A} \textbf{96}, no. 1, 95-103 (1996).

%\bibitem{Ka97} W. Kaup, On real Cartan factors, \emph{Manuscripta Math.} \textbf{92}, 191-222 (1997).

\bibitem{KaUp77} W. Kaup, H. Upmeier, Banach spaces with biholomorphically equivalent unit balls
are isomorphic, \emph{Proc. Amer. Math. Soc.} \textbf{58}, 129-133 (1978).

\bibitem{KaUp77b} W. Kaup, H. Upmeier, Jordan algebras and symmetric Siegel domains in Banach spaces, \emph{Math. Z.} \textbf{157}, 179-200 (1977).

\bibitem{NealRu12} M. Neal, B. Russo, A holomorphic characterization of operator algebras, to appear in \emph{Mathematica Scandinavica}. arXiv:1207.7353v1
    
\bibitem{Panou} D. Panou, \emph{Über die Klassifikation der beschränkten bizirkularen Gebiete in $\mathbb{C}^{n}$}, Ph.D. Dissertation, Tübingen, 1988.  

\bibitem{PeSta} A.M. Peralta, L.L. Stach{\'o}, Atomic decomposition of real ${\rm JBW}^*$-triples,
\emph{Quart. J. Math. Oxford} \textbf{52}, no. 1, 79-87 (2001).


\bibitem{Sta79}  L.L. Stachó, A short proof of the fact that biholomorphic automorphisms of the unit ball
in certain $L^p$ spaces are linear, \emph{Acta Sci. Math.} \textbf{41}, 381-383 (1979).

\bibitem{Sta82} L.L. Stachó, A projection principle concerning
biholomorphic automorphisms, \emph{Acta Sci. Math.} \textbf{44}, 99-124 (1982).

\bibitem{Sta91} L.L. Stachó, On the classification of bounded circular domains, \emph{Proc. Roy. Irish Acad.} \textbf{91}, No.2, 219-238 (1991)

\bibitem{Up87} H. Upmeier, Jordan algebras in analysis, operator theory, and quantum mechanics, CBMS, Regional conference, No. 67 (1987).

\end{thebibliography}
\end{document}